\documentclass[twoside]{article}
\usepackage{epsfig}
\usepackage{amsfonts}
\usepackage{amssymb}
\usepackage{amsmath}
\usepackage{color}
\usepackage{lineno}  

\newtheorem{theorem}{Theorem}

\newtheorem{assumption}{Assumption}{\bf}{}
{\bf}{}
{\bf}{}
{\bf}{}
\newtheorem{definition}{Definition}{\bf}{}

\def\square{\hbox{\vrule\vbox{\hrule\phantom{o}\hrule}\vrule}}
\def\qed{\rightline{$\square$}}
\newtheorem{remark}{Remark}

\title{Mixed integer predictive control and shortest path reformulation}
\author{Dario Bauso\footnote{Dipartimento di
Ingegneria Informatica, Universit\`a di Palermo, V.le delle
Scienze, 90128 Palermo, ITALY - dario.bauso@unipa.it} }

\begin{document}
\maketitle

\begin{abstract}
Mixed integer predictive control deals with optimizing integer and real control 
variables over a receding horizon. The mixed integer nature of controls might be
a cause of intractability for instances of larger dimensions. To tackle this little issue, 
we propose a decomposition method which turns the original $n$-dimensional problem into $n$ indipendent 
scalar problems of lot sizing form. Each scalar problem is then reformulated as a shortest path one and solved 
through linear programming over a receding horizon. This last reformulation step mirrors 
a standard procedure in mixed integer programming.
The approximation introduced by the decomposition can be lowered if we operate in accordance with the predictive control technique: i) optimize controls over the horizon ii) apply the first control iii) provide measurement updates of other states and repeat the procedure.
\end{abstract}

\linenumbers 

\section{Introduction}
Mixed integer predictive control arises when optimizing integer and real control variables in 
a receding horizon context~\cite{AVH2010}. For this reason, many authors see it as a specific field in the broader area 
of optimal hybrid control~\cite{BBM98}. Optimal integer control problems have been receiving a growing attention 
and are often categorized under different names. See, for instance, the literature on finite alphabet control~\cite{GQ03,TMD06}. Integer control requires a bit more than standard convex optimization techniques. 
From the literature we know that new properties come into play. As an example, look at \emph{multimodularity} presented as the counterpart of convexity in discrete action spaces~\cite{DS00}. When talking about mixed integer variables, it is, of course,  not possible not to mention the more than vast literature on mixed integer programming~\cite{NW88}. It is exactly
in this context that we have found inspiration as clarified in more details next.
 
In this paper, we have moved our steps along the line of~\cite{PW93} which surveys solution methods for mixed integer lot sizing models. Indeed, decomposing an $n$-dimensional dynamic system into $n$ indipendent lot sizing systems is almost all about this paper is centered around. The approximation introduced by the decomposition can be reduced if we operate in accordance with the predictive control technique: i) optimize controls for each indipendent system all over a prediction horizon, 
ii) apply the first control to each indipendent system, iii) provide measurement updates of other states and repeat the procedure.  The main contribution of this work is to reformulate the mixed integer problem of point i) as a shortest path problem and solve this last through linear programming. This approach mirrors the method surveyed in \cite{PW93} with the differences that here the shortest path problems run iteratively forward in time over a receding horizon. Reframing the method in a receding horizon context is an element of novelty and presents some additional and new issues which are discussed and overcome throughout the paper.

This paper differs from~\cite{AVH2010} as we focus on a smaller class of problems that can be solved exactly and do not require advanced relaxation methods which, in turn, are a main topic in~\cite{AVH2010}. To bring our discussion back to hybrid control, the lot sizing like model used here has much to do with
the inventory example briefly mentioned in~\cite{BBM98}. There, the authors simply include the example in the large list of hybrid optimal control problems but do not address the issue of how to fit general methods to this specific problem. 
On the contrary, this work cannot emphasize enough the computational benefits deriving from the ``nice structure'' of the lot sizing constraints matrix.  Binary variables, used to model impulses, match linear programming in a previous work of the same author~\cite{B09}. There, the linear reformulation is a straightforward derivation of the~\emph{(inverse) dwell time} conditions appeared first in~\cite{HLT05}. Analogies with~\cite{B09} are, for instance, the use of total unimodularity to prove the exactness of the linear programming reformulation. Differences are in the procedure itself upon which the linear program is built up. The shortest path model is an additional element which distinguishes the present approach from~\cite{B09}. 

This paper is organized as follows. We state the problem in Section~\ref{sec:problem statement}. We then move to present the decomposition method in Section~\ref{sec:robust decomposition}. In Section~\ref{sec:shortest path}, we turn to introducing the shortest path reformulation and the linear program. We dedicate the last Section~\ref{sec:numerical example} to support our theoretical analysis with some numerical results.

\section{Mixed integer predictive control}\label{sec:problem statement}
In mixed integer control we usually have continuous state $x(k) \in \mathbb R^n$, 
continuous controls $u(k) \in \mathbb R^n$ and disturbances $w(k)\in \mathbb R^n$, discrete controls $y(k) \in \{0,1\}^n$ (see e.g., \cite{AVH2010}).
Evolution of the state over a finite horizon of length $N$ is described by a linear discrete time dynamics in the general form (\ref{dynamics}), where
$A$ and $E$ are matrices of compatible dimensions: 
\begin{align}
  \label{dynamics} x(k+1)=Ax(k) + E w(k) + u(k) \geq 0, \quad x(0)=x(N)=0. 
  \end{align}
The above dynamics is characterized by one discrete and continuous control variable per each state, and this reflects the idea that we may wish to control indipendently each state component. Also, starting from initial state at zero, we wish to drive the final state to zero which is a typical requirement when controlling a system over a finite horizon. On this purpose, we have added equality constraints on the final states. Also, we force the states to remain confined within a desired region, take for it the positive orthant, which may describe a safety region in engineering applications or the desire of preventing shortcomings in inventory applications. 

Continuous and discrete controls are linked together by general \emph{capacity constraints} (\ref{capconst}), where the parameter $C$ 
is an upper bound on control:       
\begin{align}
\label{capconst} 0 \leq u(k) \leq C y(k), 	\quad y(k)\in \{0,1\}^n.
  \end{align}
For clarity reasons, $y(k)$ is the decision of controlling or not the system, and $u(k)$ is the control action. 
So if we decide not to control the system then the control action is null, otherwise this last is any value between zero and its upper bound $C$.

The following assumption helps us to describe the common situation where the disturbance seeks to push the state out of 
the desired region. 
\begin{assumption}[Unstabilizing disturbance effects]
   \begin{equation} \label{ue}E w(k) < 0.\end{equation}
\end{assumption}

At this point, the non negative nature of controls $u(k)$ should become much clearer. Actually, control actions are used to push the state far from boundaries into the positive orthant thus to counterbalance the unstabilizing effects of disturbances over a certain period to come. 
However, controlling the system has a cost and ``over acting'' on it is punished by introducing a cost/objective function as explained next. 

The objective function to minimize with respect to $y(k)$ and $u(k)$ is a linear one including proportional, holding and fixed cost terms
expressed by parameters $p^k$, $h^k$, and $f^k$ respectively: 
\begin{align}
\label{obj} \sum_{k=0}^{N-1} \left( p^k u(k) + h^k x(k) + f^k y(k)\right).
  \end{align}

Conditions (\ref{dynamics})-(\ref{obj}) introduced so far describe coincisely the problem of interest. 
In the next section, we recall a standard method to convert the problem of interest (\ref{dynamics})-(\ref{obj})  into a mixed integer
linear program returning the exact solution in terms of optimal control actions $u(k)$ and $y(k)$. 

\begin{remark} For sake of simplicity disturbances $w(k)$ are deterministic and apriori known. 
The approach presented below is still valid if we drop this assumption and turn to consider unknown disturbances. 
Only, we should carefully repropose problem (\ref{dynamics})-(\ref{obj}) in a receding horizon form with iterative measuments updates and control optimization forward in time all over the horizon.
\end{remark}

\subsection{Mixed integer linear program and exact solution.}
The mixed integer nature of the above program makes it intractable for increasing number of variables and horizon length. So, the topic 
presented below is motivated mainly by comparisons reasons and applies only to problems of relatively small dimensions.

Before introducing the mixed integer linear program we need to define the following notation.  Let us start by collecting states, continuous and discrete controls, proportional, holding and fixed costs all in opportune vectors as shown below:
$$\begin{array}{lll}x=[x(0)^T\ldots x(N)^T]^T, & u=[u(0)^T\ldots u(N-1)^T]^T, & y=[y(0)^T\ldots y(N-1)^T]^T, \\
\\p=[(p^0)^T\ldots (p^{N-1})^T]^T, & h=[(h^0)^T\ldots (h^{N-1})^T]^T, & f=[(f^0)^T\ldots (f^{N-1})^T]^T.\end{array}
$$
Furthermore, to put dynamics (\ref{dynamics}) into ``constraints'' form, let us introduce 
matrices 
$\mathbf A$, $\mathbf B$ and vector $\mathbf b$ defined as $$\mathbf A = \left[\begin{array}{cccccc} -I  & 0 & 0 & \hdots & 0 & 0 \\
A & -I & 0 & \ldots & 0 & 0 \\
0 & A & -I &  \ldots & 0 & 0 \\
0 & 0 & A & \hdots & 0 & 0 \\
\vdots & \vdots & \vdots & \ddots & \vdots & \vdots \\
0 & 0 & 0 & \ldots & A & -I \\
0  & 0 & 0 & \hdots & 0 & -I \\\end{array}\right]; \; \mathbf B = \left[\begin{array}{cccccc} 0  & 0 & \hdots & 0  \\
B & 0 & \ldots & 0 \\
0 & B & \ldots & 0  \\
\vdots & \vdots & \ddots & \vdots \\
0 & 0 & \ldots & B \\
0  & 0 & \hdots & 0  \\\end{array}\right];\; \mathbf b=\left[- \xi_0^T \, \left(E w(0)\right)^T \, \ldots \, \left(E w(N)\right)^T \,-\xi_f^T\right]^T.$$ 
Notice that once we take for $\xi_0$ and $\xi_f$ the value zero, the first and last rows in the aforementioned matrices  restate the constraints on initial and final state of (\ref{dynamics}).  

Finally, we are in the condition to establish that problem (\ref{dynamics})-(\ref{obj}) can be solved exactly through the following mixed integer
linear program:
\begin{align}
(MIPC) \quad &  \min_{u,y}  \quad  J(u,y)=p u + h x + f y \label{MIPC1}\\
  &   \mathbf A x + \mathbf B u = \mathbf b\label{eqc}\\  &  0 \leq u \leq C y, 	\quad y \in \{0,1\}^{nN}.\label{MIPC3}
  \end{align}

The mixed integer linear program (\ref{MIPC1})-(\ref{MIPC3}) is the most natural mathematical programming representation of the problem of interest 
(\ref{dynamics})-(\ref{obj}). For this reason, throughout this paper we will almost always refer to (\ref{MIPC1})-(\ref{MIPC3}) when we wish to bring back the discussion to the source problem (\ref{dynamics})-(\ref{obj}) and its exact solution. 

To overcome the intractability of the mixed integer linear program (\ref{MIPC1})-(\ref{MIPC3}), we propose a new method whose underlying idea is to bring back dynamics (\ref{dynamics}) to the lot sizing model \cite{PW93}. To do this, we introduce some additional assumptions on the structure of matrix $A$ which simplify the 
tractability and affect in no way the generality of the results. This argument is dealt with in details in the next section.

\subsection{Introducing some structure on $A$}
Our main goal in this section is to rewrite (\ref{dynamics}) in a ``nice'' form. With ``nice form'' we mean a form that 
emphasizes the analogies with standard lot sizing models \cite{PW93}. ``Stop beating around the bush'', we will henceforth refer to the 
following dynamics in state of (\ref{dynamics}):
\begin{equation}\label{dynamics1}x(k+1)=x(k) + \Delta x(k) + E w(k) + u(k) \geq 0.\end{equation}
The reasons why expression (\ref{dynamics1}) is a nice one is that it isolates the dependence of one component state on the other ones. 
To tell it differently we have separated the influence of all other states on state $i$. It will be soon clearer that 
turning our attention to the new expression (\ref{dynamics1}) is a prelude in view of the decomposition  approach
discussed later on.

Once clarified the reasons, we need next to clarify how to go from (\ref{dynamics}) to (\ref{dynamics1}) and what is the underlying assumption that allows us to do that.
Before doing this let us denote with $I \in \mathbb R^{n \times n}$ the identity matrix and $a_{ij}$ the dependence of state $i$ on state $j$.
So, we can make the following assumption.
\begin{assumption}\label{asm:1} Matrix $A$ can be decomposed as $$A=I+\Delta, \quad \quad \Delta=\left[\begin{array}{cccccc} 0 & a_{12} & \hdots &a_{1,n-1} & a_{1n} \\ a_{21} &  0 &  \hdots &a_{2,n-1} & a_{2n} \\
\vdots &  \vdots &  \ddots & \vdots & \vdots \\
a_{n1} &  a_{n2} &  \hdots & a_{n,n-1} & 0\end{array}\right].$$
\end{assumption} 
The reader may notice that (\ref{dynamics1}) is a straighforward derivation of (\ref{dynamics}) once we take for good Assumption \ref{asm:1}.

Our secondary goal in this section is to preserve the nature of the game which has stabilizing control actions playing against unstabilizing disturbances. To do this, in our next assumption we do consider the case where the influence of other states on state $i$ is relatively ``weak'' in comparison to the unstabilizing effects of disturbances.    
\begin{assumption}[Weakly coupling]\label{asm:wc}
   \begin{equation}\label{wc}\Delta x(k) + E w(k) < 0.\end{equation}
\end{assumption}
Notice that the above assumption preserves the nature of the game by bounding the effects of mutual dependence of state components 
represented by the term $\Delta x(k)$. A closer look at (\ref{ue}) and (\ref{wc}) sounds like the term $\Delta x(k)$ do not counterbalance the effects of $E w(k)$.
States mutual dependence only  emphasize or reduce ``weakly'' the unstabilizing  effects of disturbances.       

We end this section by noticing that (\ref{dynamics1}) is not yet in ``lot sizing'' form \cite{PW93}. 
In the next section, we present a decomposition approach that translate dynamics (\ref{dynamics1}) into $n$ scalar dynamics in ``lot sizing'' form \cite{PW93}.

\section{Robust decomposition}\label{sec:robust decomposition}
With the term ``decomposition'' we mean a mathematical manipulation through which the original dynamics (\ref{dynamics1}) is replaced by 
$n$ independent dynamics of the form: 
\begin{equation}\label{dynamics2} x_i(k+1)= x_i(k) - d_i(k) + u_i(k).\end{equation}
The above dynamics is in a typical lot sizing form in the sense that the (inventory) state tomorrow 
$x_i(k+1)$ is equal to the (inventory) state today $x_i(k)$ plus the discrepancy between today demand $d_i(k)$ and today reordered quantity $u_i(k)$. 
Changing (\ref{dynamics1}) with (\ref{dynamics2}) is possible once we relate the demand  $d_i(k)$ to the current values of all other state components and disturbances as expressed below:
\begin{equation}\label{d}\begin{array}{lll}  d_i(k) & = &  -\left[ \sum_{j=1, \, j \not = i}^n A_{ij}  x_j(k) + \sum_{j=1}^n  E_{ij} w_j(k)  \right]\\ & = &
- \left[\Delta_{i \bullet} x(k) + E_{i \bullet} w(k) \right]. \end{array}\end{equation}
To tell it differently, we do assume that the influence that all other states have on state $i$ enters into equation (\ref{dynamics2}) through 
demand $d_i(k)$ defined in (\ref{d}). Our next step is to make the $n$ dynamics in the form (\ref{dynamics2}) mutually independent. This is possible
by replacing the current state values $x_j(k)$, $j\not = i$ with their estimated values on the part of agent $i$ which we denote by $\tilde x_j(k)$, $j\not = i$. Still with reference to (\ref{dynamics2}), this implies to replace the current demand  $d_i(k)$ by the ``estimated'' demand  $\tilde d_i(k)$ defined as in (\ref{ed}) where $X^k$ is the set of admissible state vectors $x(k)$:
\begin{equation}\label{ed}
\tilde d_i(k)  = 
\max_{\xi \in X^k} \left\{ - \Delta_{i \bullet}  \xi - E_{i \bullet} w(k)  \right\}.\end{equation}
The idea behind (\ref{ed}) is to take for estimated value the worst admissible demand, i.e., the demand that would push the state out of the positive orthant in a fewest time and such a demand is of course the maximal one.
However, it must be noted that we cannot see any drawbacks in combining other decomposition methods  with the approach presented in the rest of the paper.     
To complete the decomposition, it is left to turn the objective function (\ref{obj}) into $n$ indipendent components 
$$J_i(u_i,y_i)=\sum_{k=0}^{N-1} \left( p_i^k u_i(k) + h_i^k x_i(k) + f_i^k y_i(k)\right).$$
Note that because of the linear structure of $J(u,y)$ in (\ref{MIPC1}), it turns $J(u,y)=\sum_{i=1}^{n} J_i(u_i,y_i)$. So, in the end we have translated our original problem into $n$ indipendent mixed integer linear minimization problems of the form (\ref{objd})-(\ref{ud}) as requested at the beginning of this section. In the spirit of predictive control, each minimization problem is then solved forwardly in time all over the horizon. So, for $\tau=0,\ldots,N-1$ we need to solve 
\begin{align}\label{objd}  \left(MIPC_i\right)  \quad &\min_{u_i,y_i} \quad  \sum_{k=\tau}^{N-1} \left( p_i^k u_i(k) + h_i^k x_i(k) + f_i^k y_i(k)\right)\\
\label{dynd}  &x_i(k+1)= x_i(k) -\tilde d_i(k) + u_i(k) \geq 0, \quad x_i(\tau)=\xi_i^0,  \, x_i(N)=0\\\label{ud} & 0 \leq u_i(k) \leq C y_i(k), 	\quad y_i(k)\in \{0,1\}.
  \end{align}

It is worth to be noted that non null initial states, which materialize in values of $\xi_i^0$ strictly greater than zero in constraints (\ref{dynd}) might induce infeasibility of $\left(MIPC_i\right)$. So, moving from $\left(MIPC\right)$ to $\left(MIPC_i\right)$ has this little drawback that we will discuss in more details later on in Section \ref{subsec:rh} together with some other issues concerned with the receding implementation of our method.


\section{Shortest path and linear programming}\label{sec:shortest path}
So far, we have first formulated the problem of interest and then decomposed it into $n$ indipendent scalar problems. 
By the way, decomposition is only the first step of our solution approach. Actually, the mixed integer nature of variables in (\ref{objd})-(\ref{ud}) is still an issue to be dealt with. This second part of the work focuses on the relaxation of the integer constraints $y_i(k)\in \{0,1\}$ which 
would facilitate the tractability of the problem. It is well known that relaxation introduces, in general, some approximation in the solution.
The main result of this work establishes that, for the problem at hand, relaxing and massaging the problem in a certain manner, will lead to 
a shortest path reformulation of the original problem. This is a great result as, it is well known that shortest path problem are in turn easily tractable and solvable through linear programming. Shortest path formulations are based on the notion of regeneration interval discussed in details in the next section.

\subsection{Regeneration interval $[\alpha,\beta]$}
Let us start by introducing a formal definition of \emph{regeneration interval} which represents the central topic in this section. The definition, available in the literature for scalar lot sizing models, is borrowed from \cite{PW93} and adapted to each single (scalar) dynamics $i$ of our decomposed $n$-dimensional model. So, with reference to the generic minimization problem $i$ expressed by (\ref{objd})-(\ref{ud}), let us state what follows.

\begin{definition}[Pochet and Wolsey 1993] A pair of periods $[\alpha, \beta]$ form a \emph{regeneration interval} for $(x_i,u_i,y_i)$
if $x_i(\alpha -1) = x_i(\beta)=0$ and $x_i(k)>0$ for $k=\alpha,\alpha+1, \ldots,\beta-1$.	\end{definition}

Given a regeneration interval $[\alpha,\beta]$, we can define the accumulated demand over the interval $d_i^{\alpha \beta}$,
and the residual demand $r_i^{\alpha \beta}$ as 
\begin{equation}\label{eq:ad}d_i^{\alpha \beta}= \sum_{k=\alpha}^{\beta} \tilde d_i(k),\quad 
r_i^{\alpha \beta}= d_i^{\alpha \beta} - \left\lfloor \frac{d_i^{\alpha \beta}}{C}\right\rfloor C.\end{equation}

Our idea is now to translate problem (\ref{objd})-(\ref{ud}) into new variables. More formally, let us consider variables $y_i^{\alpha \beta} (k)$ and $\epsilon_i^{\alpha \beta} (k)$ defined in (\ref{ye}) with the following meaning. Variable $y_i^{\alpha \beta} (k)$ is equal to one in presence of a saturated control on time $k$ and zero otherwise. Similarly, variable $\epsilon_i^{\alpha \beta} (k)$ is equal   to one in presence of a non saturated control on time $k$ and zero otherwise: \begin{equation}\label{ye}y_i^{\alpha \beta} (k)=\left\{\begin{array}{ll}1 & \mbox{if $u_i(k)=C$}\\ 0 & \mbox{otherwise.}\end{array}\right. \quad \epsilon_i^{\alpha \beta} (k)=\left\{\begin{array}{ll}1 & \mbox{if $0 < u_i(k)< C$}\\ 0 & \mbox{otherwise.}\end{array}\right.\end{equation}
To translate the meaning of $y_i^{\alpha \beta} (k)$ and $\epsilon_i^{\alpha \beta} (k)$ in a lot sizing context, such variables tell us 
on which period full or partial batches are ordered.

At this point and with in mind the above variable transformation, we can rely on well known results in the lot sizing literature which convert the original mixed integer problem (\ref{objd})-(\ref{ud}) into a number of linear programs $\left(LP_i^{\alpha \beta} \right)$, each one associated to a specific regeneration interval.  Regeneration intervals and the associated linear programs are mutually related in a way that gives raise to a shortest path problem, which will be the central topic in the next section. For now, we simply repropose below the linear programming problem associated to a single regeneration interval $[\alpha,\beta]$. 
Denoting by $e_i^k=p_i^k + \sum_{j=k+1}^{N-1} h_i^j$ and after some standard manipulation, the linear program for fixed regeneration interval $[\alpha,\beta]$ appears as:
\begin{align}
\left(LP_i^{\alpha \beta} \right)  \quad & \min_{y_i^{\alpha,\beta},u_i^{\alpha,\beta} } \quad & \sum_{k=\alpha}^{\beta} \left(  Ce_i^k + f_i^k \right) y_i^{\alpha \beta} (k) + \sum_{k=\alpha}^{\beta} \left(  r^{\alpha \beta} e_i^k + f_i^k \right) \epsilon_i^{\alpha \beta} (k) \label{objlp}\\
  &  &  \sum_{k=\alpha}^{\beta} y_i^{\alpha \beta} (k) + \sum_{k=\alpha}^{\beta} \epsilon_i^{\alpha \beta} (k) = \left\lceil \frac{d_i^{\alpha\beta}}{C} \right\rceil  \label{clp1}\\  
  && \sum_{k=\alpha}^{t} y_i^{\alpha \beta} (k) + \sum_{k=\alpha}^{t} \epsilon_i^{\alpha \beta} (k) \geq \left\lceil \frac{d_i^{\alpha t}}{C} \right\rceil, &  \quad t=\alpha,\ldots,\beta-1 \label{clp2}\\ &&  \sum_{k=\alpha}^{\beta} y_i^{\alpha \beta} (k)  = \left\lceil \frac{d_i^{\alpha\beta}- r_i^{\alpha\beta}}{C} \right\rceil
  \label{clp3}\\ 
    && \sum_{k=\alpha}^{t} y_i^{\alpha \beta} (k) \geq \left\lceil \frac{d_i^{\alpha t} - r_i^{\alpha t}}{C} \right\rceil, &  \quad t=\alpha,\ldots,\beta-1 \label{clp4}\\ 
  &&y_i^{\alpha \beta} (k), \, \epsilon_i^{\alpha \beta} (k) \geq 0, & \quad k=\alpha,\ldots,\beta.\label{clp5}
  \end{align}
The above model is extensively used in the lot sizing context. We can limit ourselves to a pair of comments on the underlying idea of 
the constraints. So, let us start by focusing on the equality constraints (\ref{clp1}) and (\ref{clp3}). These constraints tell us that 
the ordered quantity over the interval has to be equal to the accumulated demand over the same interval. This makes sense as 
initial and final state of a regeneration interval are null by definition. Let us turn our attention to the inequality constraints 
(\ref{clp2}) and (\ref{clp4}). There, we impose that the accumulated demand in any subinterval may not exceed the ordered quantity over the 
same subinterval. Again, this is due to the condition that states are nonnegative at any period of a regeneration interval.  
Finally, the objective function (\ref{objlp}) is simply a rearrangement of (\ref{objd}) induced by the variable transformation seen above and specialized to the regeneration interval $[\alpha,\beta]$ rather than on the entire horizon $[0,N]$.

We are ready to recall the following ``nice property'' of $(LP_i^{\alpha\beta})$ presented first by Pochet and Wolsey in \cite{PW93}.
\begin{theorem}[Total unimodularity] The optimal solution of $(LP_i^{\alpha\beta})$ is feasible. \end{theorem}
\begin{proof} The proof is based on the observation that the constraint matrix of $(LP_i^{\alpha\beta})$ is a $0-1$ matrix.
We can reorder the constraints in a certain manner, so that matrix has the consecutive 1's property on each column and turns to be 
totally unimodular. It follows that $y_i^{\alpha,\beta}$ and $\epsilon_i^{\alpha,\beta}$ are $0-1$ in any extreme solution.  
\qed \end{proof} 
The above theorem represents a first step in the process of converting the mixed integer problem $(MIPC_i)$ into a linear programming one.

\subsection{Shortest path}
In the previous section we have introduced a linear programming problem associated to a specific regeneration interval.
In this section, we resort to well known results on lot sizing to come up with a shortest path model which links together the linear programming problems of all possible regeneration intervals. Actually, it must be noted that the solution of (\ref{objd}) -(\ref{ud}) can be expressed as a unique regeneration interval $[0,N]$ or as a list of regeneration intervals.  

So, let us define variables $z_i^{\alpha \beta} \in \{0,1\}$ which tell us one or zero whenever a regeneration interval $[\alpha,\beta]$ appears or not in the solution of (\ref{objd}) -(\ref{ud}). The linear programming problem solving (\ref{objd}) -(\ref{ud}) takes on the form below. For $\tau=0,\ldots,N-1$, solve

\begin{align}
\label{objsp}\left(LP_i\right)  \quad & \min_{y_i^{\alpha\beta},u_i^{\alpha\beta}, z_i^{\alpha\beta}} \quad & \sum_{\alpha=\tau+1}^{N-1} \sum_{\beta=\alpha}^{N-1} \sum_{k=\alpha}^{\beta} \left[ \left(  Ce_i^k + f_i^k \right) y_i^{\alpha \beta} (k) + \sum_{k=\alpha}^{\beta} \left(  r^{\alpha \beta} e_i^k + f_i^k \right) \epsilon_i^{\alpha \beta} (k) \right]\end{align}
\begin{align}
&& \sum_{\beta=\tau+1}^N z_i^{\tau+1\beta}=1
\label{csp0}\\
&& \sum_{\alpha=\tau+1}^{t-1} z_i^{\alpha,t-1} - \sum_{\beta=t}^N z_i^{t\beta}=0 & \quad t=\tau+2,\ldots,N,  \quad \tau+1 \leq \alpha \leq \beta \leq N
\label{csp1}\\
  &  &  \sum_{k=\alpha}^{\beta} y_i^{\alpha \beta} (k) + \sum_{k=\alpha}^{\beta} \epsilon_i^{\alpha \beta} (k) = \left\lceil \frac{d_i^{\alpha\beta}}{C} \right\rceil z_i^{\alpha\beta},& \quad \tau+1 \leq \alpha \leq \beta \leq N \label{csp2}\\  
  && \sum_{k=\alpha}^{t} y_i^{\alpha \beta} (k) + \sum_{k=\alpha}^{t} \epsilon_i^{\alpha \beta} (k) \geq \left\lceil \frac{d_i^{\alpha t}}{C} \right\rceil z_i^{\alpha \beta}, &  \quad t=\alpha,\ldots,\beta-1, \quad \tau+1 \leq \alpha \leq \beta \leq N \label{csp3}\\ &&  \sum_{k=\alpha}^{\beta} y_i^{\alpha \beta} (k)  = \left\lceil \frac{d_i^{\alpha\beta}- r_i^{\alpha\beta}}{C} \right\rceil
z_i^{\alpha \beta} &  \quad \tau+1 \leq \alpha \leq \beta \leq N  \label{csp4}\\ 
    && \sum_{k=\alpha}^{t} y_i^{\alpha \beta} (k) \geq \left\lceil \frac{d_i^{\alpha t} - r_i^{\alpha t}}{C} \right\rceil z_i^{\alpha\beta}, &  \quad t=\alpha,\ldots,\beta-1,  \quad \tau+1 \leq \alpha \leq \beta \leq N \label{csp5}\\ 
  &&y_i^{\alpha \beta} (k), \, \epsilon_i^{\alpha \beta} (k),\, z_i^{\alpha \beta} \geq 0, & \quad k=\alpha,\ldots,\beta.\label{csp6}
  \end{align}

Let us spend a couple of words on the meaning of the above linear program. Constraints (\ref{csp2})-(\ref{csp6}) should be familiar to the reader as 
they already appeared in (\ref{clp1})-(\ref{clp5}). The only difference is that, now, because of the presence of $z_i^{\alpha \beta}$ in the right hand term, the constraints referring to a given regeneration interval come into play only  if that interval is chosen as part of the solution, that is, whenever $z_i^{\alpha \beta}$ is set equal to one. Furthermore, a new class of constraints appear in (\ref{csp0})-(\ref{csp1}). These constraints are typical of shortest path problems and in this specific case help us to force the variables  $z_i^{\alpha \beta} (k)$ to describe a path from $0$ to $N$. Finally, note that for $\tau=0$, the linear program $(LP_i)$ coincide with the linear program 
presented by Pochet and Wolsey in \cite{PW93}.

At this point, we are in a position to recall the crucial result established in \cite{PW93}. 
\begin{theorem}[Pochet and Wolsey, 1993] 
The linear program $(LP_i)$ solves $(MIPC_i)$.  \end{theorem}
\begin{proof} (Sketch) It turns out that the linear program $(LP_i)$ is a 
shortest path problem on variables $z_i^{\alpha,\beta}$. Arcs are all associated to a different regeneration interval $[\alpha,\beta]$
and the respective costs are the optimal values of the objective functions of the corresponding linear programs $(LP_i^{\alpha,\beta})$.    
We refer the reader to \cite{PW93} for further details. 

\qed  \end{proof}

\subsection{Receding horizon implementation of $(LP_i)$}\label{subsec:rh}
This section is dedicated to certain issues concerned with the implementation of $(LP_i)$ in a receding horizon context as typical of predictive control. 
As the reader may know, in predictive control we solve $(LP_i)$ iteratively and forward in time all over the horizon. 
In the formulation of $(LP_i)$, this is stated clearly when we specify that $\tau$ goes from $0$ to $N-1$ and for each value of 
$\tau$ we obtain a new linear program of type $(LP_i)$. After we solve $(LP_i)$ for $\tau=0$, we apply the first control to the system,
update initial states according to the last available measurements at time $\tau=1$ and move to solve a new $(LP_i)$ starting at $\tau=1$. 
We repeat this procedure until the end of the horizon, $\tau=N-1$. So, consecutive linear programs are linked together by 
initial state condition expressed in (\ref{dynd}), and which we rewrite below 
$$x_i(\tau)=\xi_i^0.$$
At this point, we would restate with emphasis the fact that dealing with non null initial states is a main difference between the linear program $(LP_i)$ and the linear program used in the lot sizing literature \cite{PW93}. 
To counter this little issue, we need to elaborate more on how to compute the accumulated demand in (\ref{eq:ad}). 
Actually, take for $[\tau,t]$ any interval with $x(\tau)=\xi_i^0 >s 0$. Then, condition (\ref{eq:ad}) needs to be revised as
\begin{equation} \label{rf}d_i^{\tau t}= \max\left\{\sum_{k=\tau}^{t} \tilde d_i(k)- \xi_i^0,0\right\}.\end{equation}
The rational behind the above formula has an immediate interpretation in the lot sizing context.
Actually, the effective demand over an interval is the accumulated demand reduced by the inventory stored and initially available at the warehouse.
From a computational standpoint, the revised formula (\ref{rf}) has a different effect depending on the cases where
the accumulated demand exceeds the initial state or not as discussed next.
\begin{enumerate}
\item $\sum_{k=\alpha}^{\beta} \tilde d_i(k) \geq \xi_i^0$:
the mixed linear program $(MPC_i)$ with initial state $x(\tau)=\xi_i^0 > 0$ and accumulated demand $\sum_{k=\alpha}^{\beta} \tilde d_i(k)$ is turned into an $(LP_i)$ characterized by null initial state $x(\alpha-1)=0$ and effective demand $d_i^{\alpha \beta}=\sum_{k=\alpha}^{\beta} \tilde d_i(k) - \xi_i^0$ as in the example below:
\begin{eqnarray*}(MPC_i) \quad \sum_{k=\alpha}^{\beta} \tilde d_i(k)=12, \quad x(\tau)=\xi_i^0=10 &  \Longrightarrow &  
(LP_i) \quad  x(\alpha-1)=0,  \quad d_i^{\alpha \beta}=2;\end{eqnarray*}

\item $\sum_{k=\alpha}^{\beta} \tilde d_i(k) < \xi_i^0$:
the mixed linear program $(MPC_i)$ with initial state $x(\tau)=\xi_i^0 > 0$ and accumulated demand $\sum_{k=\alpha}^{\beta} \tilde d_i(k)$ is unfeasible. The solution obtained at previous period $\tau-1$ applies. A second example is shown next: 
\begin{eqnarray*}(MPC_i) \quad \sum_{k=\alpha}^{\beta} \tilde d_i(k)=7, \quad x(\tau)=\xi_i^0=10 &  \Longrightarrow &  
(LP_i) \text{ unfeasible.} \end{eqnarray*}

In both cases, the revised formula (\ref{rf}) helps us to generalize the linear program $(LP_i)$ to cases where the initial state is non null and this is a crucial point when applying the lot sizing model in a receding horizon form.

\end{enumerate}
\section{Numerical example}\label{sec:numerical example}
In this specific example, dynamics (\ref{dynamics}) takes on the form expressed below. Such a dynamics is particularly significative 
as it reproduces the typical influence between position and velocity in a sampled second-order system. Initial and final states are null and state values must remain in the positive quadrant all over the horizon. More specifically, denoting by $x_1$ the position
and $x_2(k)$ an opposite in sign velocity, the dynamics appears as:
\begin{equation}\label{exdyn}\left[\begin{array}{ll} x_1(k+1)\\x_2(k+1)\end{array}\right]=\left[\begin{array}{lc} 1 & - \kappa\\\kappa & 1\end{array}\right] \left[\begin{array}{ll} x_1(k)\\x_2(k)\end{array}\right] - \left[\begin{array}{ll} w_1(k)\\w_2(k)\end{array}\right]+\left[\begin{array}{ll} u_1(k)\\u_2(k)\end{array}\right]\geq 0, \quad \left[\begin{array}{ll} x_1(0)\\x_2(0)\end{array}\right]=\left[\begin{array}{ll} x_1(N)\\x_2(N)\end{array}\right]=0. \end{equation}
A closer look at the first equation reveals that a greater velocity $x_2(k)$ reflects into a faster decrease of position $x_1(k+1)$. Similarly, the second equation tells us that a greater position $x_1(k)$ induces a faster increase of velocity $x_2(k+1)$ because of some elastic reaction.
In both equations, the non negative disturbances $w_i(k) \leq 0$ seek to push the states $x_i(k)$ out of the positive quadrant in accordance to Assumption \ref{ue}. Their effect is counterbalanced by positive control actions $u_i$.  Notice that matrix $A$ can be decomposed as described in 
Assumption \ref{asm:1}. Also, acting on parameter $\kappa$ we can easily guarantee the ``weakly coupling'' condition expressed in Assumption \ref{asm:wc}.

Turning to the capacity constraints (\ref{capconst}), for this two-dimensional example, these constraints can be rewritten as:
$$0 \leq \left[\begin{array}{ll} u_1(k)\\u_2(k)\end{array}\right] \leq C \left[\begin{array}{ll} y_1(k)\\y_2(k)\end{array}\right], \quad 
\left[\begin{array}{ll} y_1(k)\\y_2(k)\end{array}\right]\in \{0,1\}^2.$$
It is left to comment on the objective function (\ref{obj}). We consider the case where fixed costs are much more relevant than proportional and holding ones. This materializes in choosing a high value for $f^k$ in comparison to values of parameters $p^k$, $h^k$ as shown in the next linear
objective function: 
$$J(u,y)=   \sum_{k=0}^{N-1} \left( \mathbf 1^n u(k) + \mathbf 1^n x(k) + \mathbf {100}^n y(k)\right).$$
This choice makes sense for 
two reasons. First, all the work is centered around issues deriving from the integer nature of $y(k)$. So, high values of $f^k$ emphasize the
role of integer variables in the objective function. Second, high fixed costs incentivate solutions with the fewest number of control actions and
this facilitate the validation and interpretation of the simulated results. 

The next step is to decompose dynamics (\ref{exdyn}) in scalar lot sizing form (\ref{dynd}) which we rewrite below:
$$x_i(k+1)=x_i(k) - \tilde d_i(k) + u_i(k).$$
When it comes to the discussion on how to compute the estimated demand $\tilde d_i$, a natural choice is to set $\tilde d_i$ as below, where
we have denoted by $\tilde x_1(k)$ (respectively $\tilde x_2(k)$) the estimated value of state $x_1(k)$ (respectively $x_2(k)$) available
to agent $2$ (agent $1$):      
\begin{equation}\label{td}\left[\begin{array}{ll} \tilde d_1(k)\\ \tilde d_2(k) \end{array}\right]=\left[\begin{array}{cc} 0 & \kappa\\ -\kappa & 0 \end{array}\right] \left[\begin{array}{ll} \tilde x_1(k)\\ \tilde x_2(k)\end{array}\right] +  \left[\begin{array}{ll} w_1(k)\\w_2(k)\end{array}\right].\end{equation}
Now, the question is: which expression should we use to represent 
the set  of admissible state vectors $X^k$ appearing in equation (\ref{ed})? This question has much to do with another one: 
how does agent 1 predict $\tilde x_2$ and the same for agent 2 with respect to state $\tilde x_1$?
A possible answer is shown next:
\begin{equation}\label{est}\left[\begin{array}{ll} \tilde x_1(k+1)\\ \tilde x_2(k+1)\end{array}\right]=\left[\begin{array}{ll} \tilde x_1(k)\\ \tilde x_2(k)\end{array}\right] +\left[\begin{array}{ll} 0 \\ \kappa \bar x_1\end{array}\right]  - \left[\begin{array}{ll} 0 \\ w_2(k) \end{array}\right]+ \left[\begin{array}{ll} 0 \\ C \end{array}\right] ,\quad \left[\begin{array}{ll} \tilde x_1(0)\\ \tilde x_2(0)\end{array}\right]= \left[\begin{array}{ll} x_1(0)\\ \tilde x_2(0)\end{array}\right].\end{equation}
Let us elaborate more on the above equations. Regarding to variable $\tilde x_2(k)$, this is used in the evolution of
$\tilde d_1(k)$ as in the first equation of (\ref{td}). Because of the positive contribution of the term $\kappa \tilde x_2(k)$ on $\tilde d_1(k)$,
a conservative approach would suggest to take for $\tilde x_2(k)$ a possible upper bound of $x_2(k)$ and this is exactly the spirit behind 
the evolution of $\tilde x_2(k)$ as expressed in the second equation of (\ref{est}). Here, $\bar x_1$ is an average value for $ x_1$. A similar reasoning applies to $\tilde x_1(k)$, used in the evolution of $\tilde d_2(k)$ as in the second equation of (\ref{td}). We now observe a negative contribution of the term $-\kappa \tilde x_1(k)$ on $\tilde d_2(k)$ and therefore take for $\tilde x_1(k)$ a possible lower bound of $x_1(k)$ as shown in the first equation of (\ref{est}).

We can now move to show and comment our simulated results. We have carried out two different set of experiments whose parameters are displayed in 
Table \ref{t:data1}. In the line of the weakly coupling assumption (see Assumption \ref{asm:wc}), 
we have set $\kappa$ small enough and in the range equal from $0.01$ to $0.225$. Such a range works good as we will see that $|\kappa x_i|$ is always less than $w_i$, which also means $\Delta x(k) + E w(k) < 0$. For sake of simplicity and without loss of generality, capacity $C$
is set to three, disturbances $w_i$ are unitary and $\bar x_1$ is equal to one. Unitary disturbances facilitate the check out and interpretation
of the results as when the accumulated demand over the horizon turns to be very close to the horizon length.  
The two experiments differ also in the horizon length $N$
for the reasons clarified next. 

The first set of experiments aims at analysing the computational benefits of decomposition and relaxation upon which our solution method is based. 
So, we consider horizon lenghts $N$ from one to ten. We do not need to consider larger values of $N$ as even in this small range of values, differences
in the computational times are already evident enough as clearly illustrated in Fig. \ref{fig:Time22}. 
Here, we plot the average computational time vs. the horizon length $N$ of the mixed integer predictive control problem (solid diamonds), of the decomposed problem $(MIPC_i)$ (dashed squares), and of the linear program $(LP_i)$. Average computational time means the average time for one agent to make a single decision (the total time is about $2N$ times the average one). As the reader may notice, the 	computational time of the linear program $(LP_i)$ is a fraction either of the one requested by the $(MPC)$ or of the one required by the $(MIPC_i)$.  
\begin{table}
\begin{center}
   \begin{tabular}{|c|c|c|c|c|c|c|c|c|c|c|c|}
  \hline
   & $N$ & $\kappa$ &  $C$ &  $w_1(k)$ &  $w_2(k)$ & $\bar x_1$   \\\hline
   I & 1 \ldots 10& 0.1   &  3   &   1      &  1   &  1    \\\hline
   II & 6 & $\{0.01, \;0.2, \:0.225\}$   &  3   &   1      &  1           &  1 \\\hline
   
  \end{tabular}\\
  \end{center}
  \caption{Simulation parameters chosen for the two experiments.}\label{t:data1}\end{table}

\begin{center}
(Figure \ref{fig:Time22} about here)
\end{center}

In a second set of simulations, we have inspected how the percentage error $$\epsilon \%=\frac{\text{optimal cost of $(MPC_i)$}- \text{optimal cost of $(MPC)$}}{\text{optimal cost of $(MPC)$}}\%$$ varies with different values of the elastic coefficient $\kappa$. The role of $\kappa$ is crucial as we recall that $\kappa$ describes the effective tightness and coupling between different states $x_1(k)$ and $x_2(k)$. 
We do expect that  small values for coefficient $\kappa$, which means weak coupling of state components, may lead to small errors $\epsilon \%$. Differently, high values of $\kappa$, describing a strong coupling between state components, are supposed to induce higher values of $\epsilon \%$.     

This is in line with what we can observe in Fig. \ref{fig:Fig3} where we plot the error $\epsilon \%$ as function of coefficient $\kappa$. For a relatively small values of $\kappa$ in the range from $0$ to $0.2$, we observe a percentage error not exceeding the one percent, $\epsilon \% \leq 1$. A discountinuity at around $\kappa=0.2$ causes the error $\epsilon \%$ to go from about $1 \%$ to $20 \%$. 

\begin{center}
(Figure \ref{fig:Fig2} about here)\end{center}

We might not be surprised as discountinuity of errors is typical in mixed  integer programs and we try to clarify this in more details in the plot of Fig. \ref{fig:Fig4}. 
Here, for a horizon length $N=6$ and for a relatively high value of $\kappa=0.225$, we display the exact solution (dashed squares) and approximate solution (solid triangles) returned by the mixed integer linear program $(MIPC)$ and by the linear program $(LP_i)$ respectively.  The  solution is in terms of the time plot of states $x_i(k)$, continuous controls $u_i(k)$ and discrete controls $y_i(k)$. Dotted lines represent predicted trajectories in earlier periods of the receding horizon implementation. At a first check, and this is in accordance with what we 
do expect, we note that controls $u_i(k)$ never exceed the capacity and are always associated to unitary control actions $y_i(k)$. Now, with a look at the behaviour of discrete controls $y_1(k)$, it can be observed that the approximate solution
presents four control actions (four peaks  at one), whereas the exact solution has control $y_1(k)$ acting on the system only three times (three peaks at one). One peak out of four represents an increase in the use of control actions of about $25$ percent which reflects into an approximate increase in the percentage error of $20 \%$. A last observation concerning the exact plot of $y_i(k)$ is that the number of control actions are as minimal as possible, i.e., three for $y_1(k)$ and two for $y_2(k)$. This makes sense as the accumulated demand over the horizon approximates by above the horizon length. This implies that the minimum number of control actions can be roughly obtained dividing the accumulated demand (about something above six) by the capacity $C$ (equal to three) and rounding the fractional result up to the next integer. 
\begin{center}
(Figure \ref{fig:Fig3} about here)\end{center}
Let us move to compare exact and approximate solutions for a smaller value of $\kappa=0.2$. 
With reference to Fig. \ref{fig:Fig4}, we observe that, differently from above, discrete controls $y_i(k)$ coincide.
However, we still have notable differences in the plot of continuous controls $u_1(k)$ which cause distinct state trajectories
for $x_1(k)$. Small differences can be noted for $u_2(k)$ and $x_2(k)$ as well. 
The observed differences still cause a reduced percentage error $\epsilon \%=1$.     
\begin{center}
(Figure \ref{fig:Fig4} about here)\end{center}
We conclude our simulations by showing that the percentage error $\epsilon \%$ is around zero when we reduce further the value of $\kappa$ to $0.01$. 
This is evident if we look at Fig. \ref{fig:Fig5}, where plots of different styles overlap which means that exact and approximate solutions coincide. 
\begin{center}
(Figure \ref{fig:Fig5} about here)\end{center}

\newpage

\begin{figure}
\centering
\includegraphics[width=12cm]{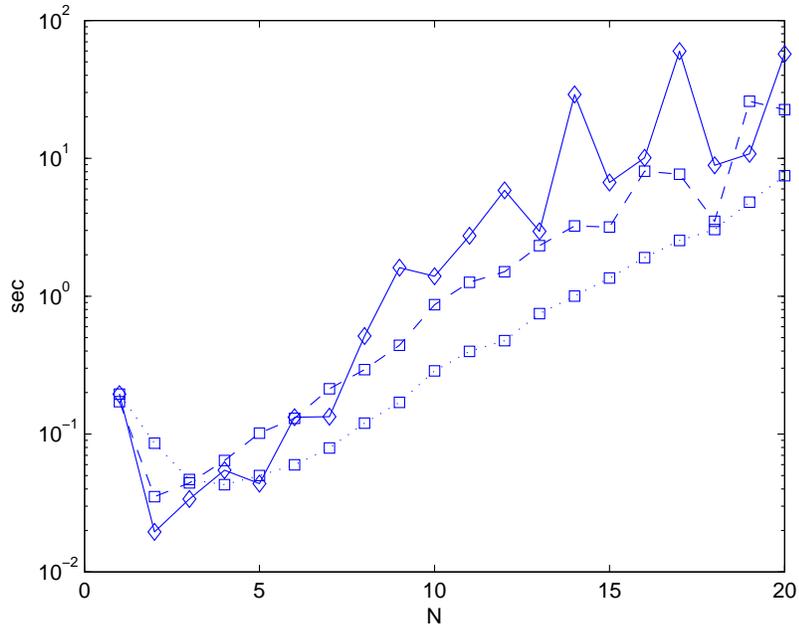}
\caption{Average computational time vs. horizon length $N$ of the mixed integer predictive control problem (solid diamonds), of the
decomposed problem $(MIPC_i)$ (dashed squares), and of the linear program $(LP_i)$.}
\label{fig:Time22}
\end{figure}

\begin{figure}
\centering
\includegraphics[width=12cm]{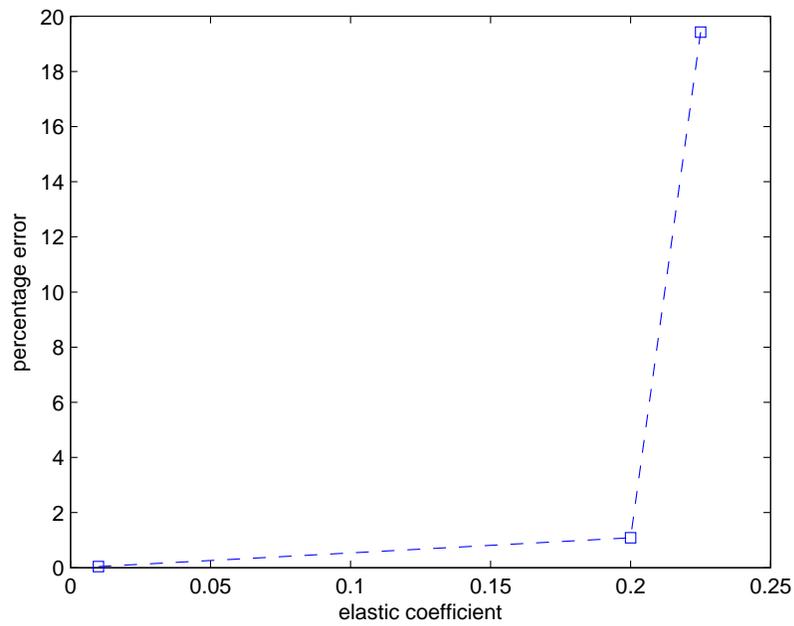}
\caption{Percentage error $\epsilon \%$ for different values of the elastic coefficient $k$.}
\label{fig:Fig2}
\end{figure}

\begin{figure}
\centering
\includegraphics[width=12cm]{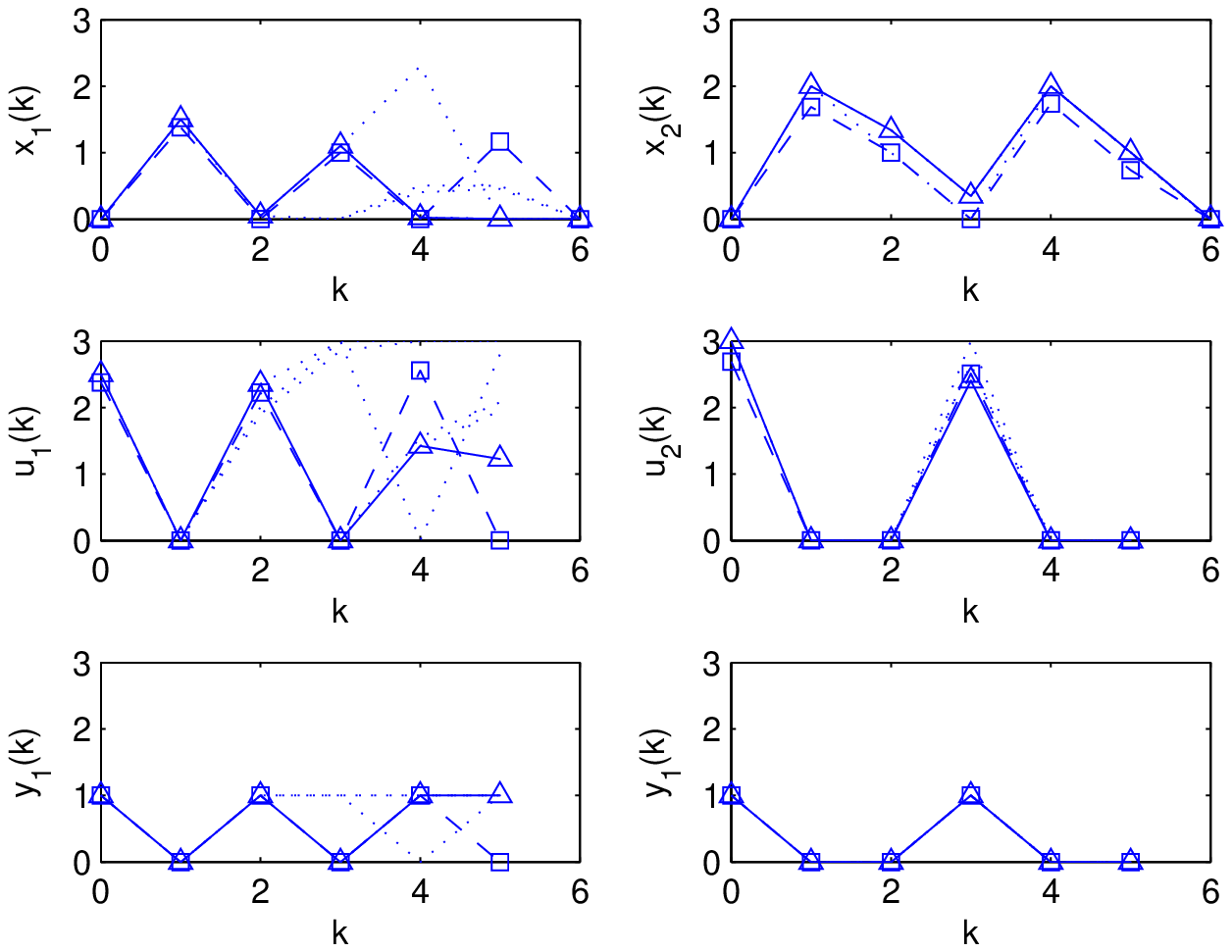}
\caption{Elastic coefficient $\kappa=0.225$. Exact solution (dashed squares) and approximate solution (solid triangles) returned by the mixed integer linear program $(MIPC)$ and by the linear program $(LP_i)$ respectively. Horizon length $N=6$. Time plot of states $x_i(k)$, continuous controls $u_i(k)$ and discrete controls $y_i(k)$.}
\label{fig:Fig3}
\end{figure}

\begin{figure}
\centering
\includegraphics[width=12cm]{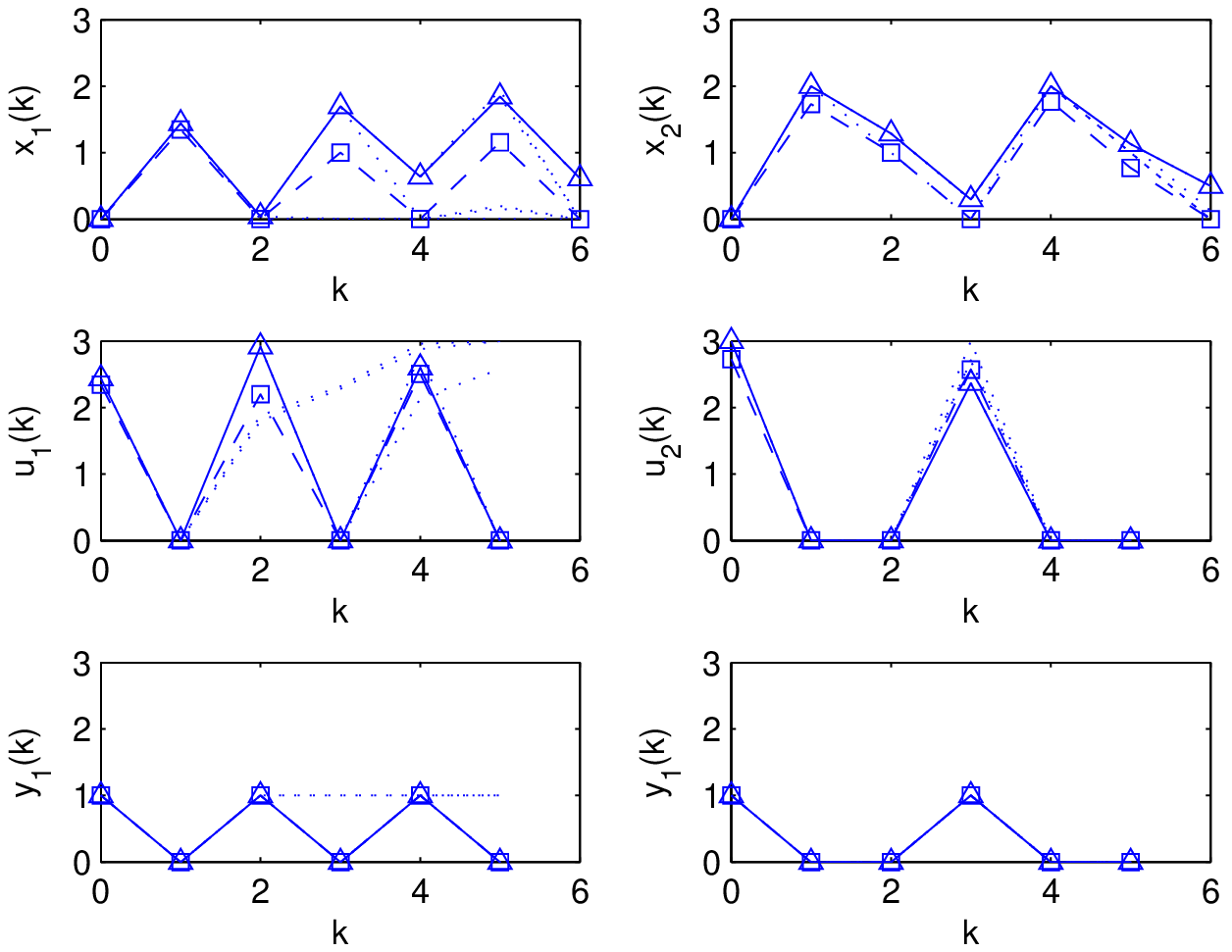}
\caption{Elastic coefficient $\kappa=0.20$. Exact solution (dashed squares) and approximate solution (solid triangles) returned by the mixed integer linear program $(MIPC)$ and by the linear program $(LP_i)$ respectively. Horizon length $N=6$. Time plot of states $x_i(k)$, continuous controls $u_i(k)$ and discrete controls $y_i(k)$.}
\label{fig:Fig4}
\end{figure}

\begin{figure}
\centering
\includegraphics[width=12cm]{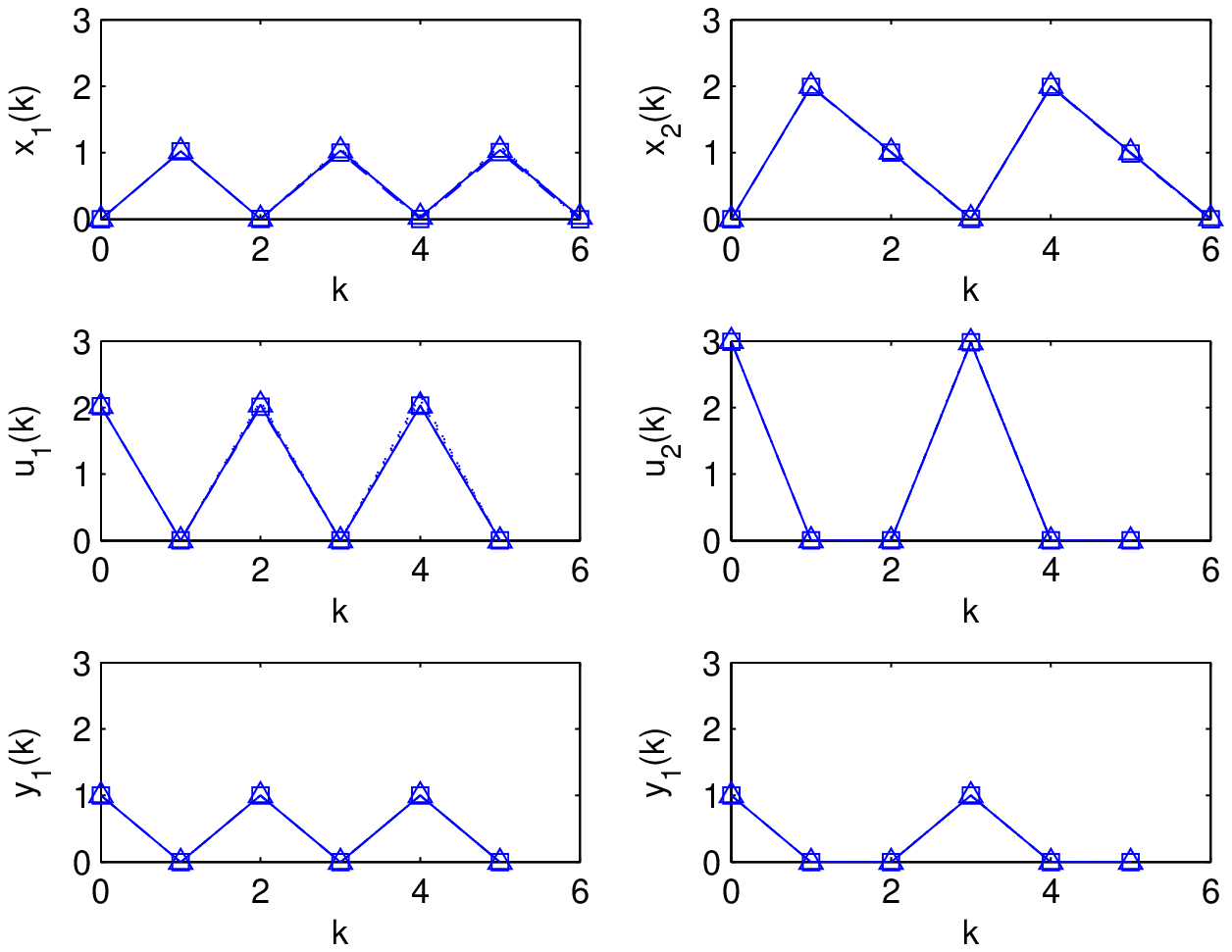}
\caption{Elastic coefficient $\kappa=0.001$. Exact solution (dashed squares) and approximate solution (solid triangles) returned by the mixed integer linear program $(MIPC)$ and by the linear program $(LP_i)$ respectively. Horizon length $N=6$. Time plot of states $x_i(k)$, continuous controls $u_i(k)$ and discrete controls $y_i(k)$.}
\label{fig:Fig5}
\end{figure}

\end{document}